\documentclass[12pt]{amsart}

\usepackage{amssymb}

\usepackage{enumerate}

\usepackage{graphicx}

\makeatletter
\@namedef{subjclassname@2010}{%
  \textup{2010} Mathematics Subject Classification}
\makeatother

\newtheorem{thm}{Theorem}[section]

\newtheorem{lemma}[thm]{Lemma}

\theoremstyle{definition}

\def\Z{\mathbb Z}

\def\pmod #1{\ ({\rm mod}\ #1)}
\def\le{\leqslant}
\def\ge{\geqslant}

\numberwithin{equation}{section}

\begin{document}


\baselineskip=17pt



\title[The sum of divisors of a quadratic form]
{The sum of divisors of a quadratic form}

\author[L. Zhao]{Lilu  Zhao}
\address{School of Mathematics, Hefei University of Technology, Heifei 230009, People's Republic of China}
\email{zhaolilu@gmail.com}

\date{}

\begin{abstract}
We study the sum of divisors of the quadratic form
$m_1^2+m_2^2+m_3^2$. Let $$S_3(X)=\sum_{1\le m_1,m_2,m_3\le
X}\tau(m_1^2+m_2^2+m_3^2).$$ We obtain the asymptotic formula
$$S_3(X)=C_1X^3\log X+ C_2X^3+O(X^2\log^7 X),$$
where $C_1,C_2$ are two constants. This improves upon the error
term $O_\varepsilon(X^{8/3+\varepsilon})$ obtained by Guo and Zhai
\cite{GZ}.
\end{abstract}



\maketitle
\section{Introduction}

Let $\tau(n)$ denote the number of divisors of $n$. Gafurov
\cite{Gafurov1,Gafurov2} investigated the sum
\begin{align*}S_2(X)=\sum_{1\le m,n\le
X}\tau(m^2+n^2),
\end{align*}
and obtained the asymptotic formula
\begin{align*}S_2(X)=A_1X^2\log X+ A_2X^2+O(X^{5/3}\log^9 X),\end{align*}
where $A_1$ and $A_2$ are certain constants. This was improved by
Yu \cite{Yu}, who proved that
\begin{align*}S_2(X)=A_1X^2\log X+ A_2X^2+O_\varepsilon(X^{3/2+\varepsilon}),\end{align*}
for any fixed positive real number $\varepsilon$.

C. Calder\'{o}n and M. J. de Velasco \cite{CV} studied the
divisors of the quadratic form $m_1^2+m_2^2+m_3^2$. Let
\begin{align}\label{e11}S_3(X)=\sum_{1\le m_1,m_2,m_3\le
X}\tau(m_1^2+m_2^2+m_3^2).\end{align} C. Calder\'{o}n and M. J.
Velasco established the asymptotic formula
\begin{align}\label{e12}S_3(X)=\frac{8\zeta(3)}{5\zeta(5)}X^3\log X+ O(X^3).\end{align}
Let
\begin{align}\label{e13}\mathfrak{S}_1=\sum_{q=1}^\infty\frac{1}{q^4}\sum_{\substack{a=1\\ (a,q)=1}}^q
\Big(\sum_{x=1}^qe\big(\frac{ax^2}{q}\big)\Big)^3\end{align}and
\begin{align}\label{e14}\mathfrak{S}_2=\sum_{q=1}^\infty\frac{(2\gamma-2\log q)}{q^4}\sum_{\substack{a=1\\ (a,q)=1}}^q
\Big(\sum_{x=1}^qe\big(\frac{ax^2}{q}\big)\Big)^3.\end{align}
Define
\begin{align*}\mathcal{I}_1=\int_{-\infty}^{\infty}\Big(\int_0^1e(x^2\beta)dx\Big)^3
\Big(\int_0^3e(-x\beta)dx\Big)d\beta,\end{align*}and
\begin{align*}\mathcal{I}_2=\int_{-\infty}^{\infty}\Big(\int_0^1e(x^2\beta)dx\Big)^3
\Big(\int_0^3(\log x)e(-x\beta)dx\Big)d\beta.\end{align*}
Recently, Guo and Zhai \cite{GZ} improved (\ref{e12})
to\begin{align*}S_3(X)=\frac{8\zeta(3)}{5\zeta(5)}X^3\log
X+(\mathfrak{S}_1\, \mathcal{I}_2+\mathfrak{S}_2\,
\mathcal{I}_1)X^3+ O_\varepsilon(X^{8/3+\varepsilon}),\end{align*}
where $\varepsilon$ is an arbitrary positive number.

The purpose of this paper is to prove the following result.
\begin{thm}\label{theorem}
Let $S_3(X)$ be defined in (\ref{e11}). We
have\begin{align}\label{e15}S_3(X)=\frac{8\zeta(3)}{5\zeta(5)}X^3\log
X+(\mathfrak{S}_1 \mathcal{I}_2+\mathfrak{S}_2 \mathcal{I}_1)X^3+
O(X^{2}\log^7X).\end{align}
\end{thm}

It is worth to pointing out that the error term $O(X^2\log^7 X)$
in (\ref{e15}) cannot be replaced by $o(X^2\log X)$. Otherwise, on
considering $X=N+1/3$ and $X=N+2/3$ with $N$ a positive integer,
we have for $j=1,2$ that
\begin{align*}S_3(N+j/3)=C_1(N+j/3)^3\log(N+j/3)+
C_2(N+j/3)^3+ o(N^{2}\log N).\end{align*} This leads to a
contradiction since $S_3(N+\frac13)=S_3(N+\frac23)$. Therefore,
the room for the further improvement is no more than $O(\log^6X)$.

 \vskip3mm

\section{Preliminaries}
Throughout this paper, we assume that $X$ is a sufficiently large
real number. Let $P=\lfloor 5X\rfloor$, where $\lfloor x\rfloor$
denotes the greatest integer not exceeding real number $x$.

Our initial approach is to install the smooth weight in our
problem. Let $w_0(x)$ be a function on $[0,\infty)$ satisfying
$w_0(x)=1$ for $x=[0,3]$, $w_0(x)=0$ for $x\ge 4$ and $w_0''(x)$
exists for $x>0$. Let $w_1(x)$ be a function supported on
$[1/2,3]$ satisfying $w_1(x)=1$ for $1<x<2$ and $w_1''(x)$ exists
for $0<x<2$. Now we introduce the function $w(x)$ defined by
\begin{align*}w(x)=\begin{cases}w_1(x), &\ \ \textrm{ if } \ x<1
\\ 1, &\ \ \textrm{ if } \ 1\le x\le 3X^2,
\\ w_0(x/X^2), &\ \ \textrm{ if } \ x>3X^2.\end{cases}\end{align*}
Note that $w(x)$ is a smooth function supported on $[1/2,4X^2]$.
By (\ref{e11}),
\begin{align*}S_3(X)=\sum_{1\le m_1,m_2,m_3\le X}\tau(m_1^2+m_2^2+m_3^2)=\sum_{\substack{1\le m_1,m_2,m_3\le X
\\ 1\le n\le 3X^2 \\ m_1^2+m_2^2+m_3^2=n}}\tau(n).\end{align*}
Since $m_1,m_2,m_3\le X$ and $m_1^2+m_2^2+m_3^2=n$ together imply
$n\le 3X^2$, we conclude that
\begin{align*}S_3(X)=\sum_{\substack{1\le m_1,m_2,m_3\le X
\\ n \\ m_1^2+m_2^2+m_3^2=n}}\tau(n)w(n).\end{align*}
In order to apply the circle method, we introduce the exponential
sums\begin{align*}f(\alpha)=\sum_{1\le m\le X}e(m^2\alpha),\ \ \
h(\alpha)=\sum_{n}w(n)\tau(n)e(n\alpha),
\end{align*}where we write $e(z)$ for $e^{2\pi iz}$. By
orthogonality,
\begin{align}\label{e21}S_3(X)=\int_0^1f(\alpha)^3h(-\alpha)d\alpha.\end{align}

Let $g$ be a smooth, compactly supported function on $(0,\infty)$.
The Voronoi type summation formula asserts that
\begin{align}\sum_{n=1}^\infty\tau(n)e(\frac{an}{q})g(n)=&q^{-1}\int(\log x+2\gamma-2\log
q)g(x)dx\notag
\\ &+\sum_{n=1}^\infty\tau(n)e(-\frac{\overline{a}\,n}{q})\widetilde{g}_{Y}(n)
+\sum_{n=1}^\infty\tau(n)e(\frac{\overline{a}\,n}{q})\widetilde{g}_{K}(n)\label{e22}\end{align}
where $(a,q)=1$, $\overline{a}$ denotes the inverse of $a$ modulo
$q$, and
\begin{align*}\widetilde{g}_{Y}(y)=-\frac{2\pi}{q}\int_0^\infty g(x)Y_0\big(\frac{4\pi\sqrt{xy}}{q}\big)dx,
\ \ \widetilde{g}_{K}(y)=\frac{4}{q}\int_0^\infty
g(x)K_0\big(\frac{4\pi\sqrt{xy}}{q}\big)dx.\end{align*} Here $Y_0$
and $K_0$ denote the standard Bessel functions. For the proof of
the above well-known formula, one may refer to Section 4.5 of
Iwaniec and Kowalski \cite{IK}. For $0<x\ll 1$, one has
\begin{align}\label{e23}Y_0(x)\ =\ \frac{2}{\pi}\log\frac{x}{2}+O(1),\ \ \ \ K_0(x)=\log\frac{2}{x}+O(1)\end{align}
and
\begin{align}\label{e24}Y_\nu(x),\ K_\nu(x)\ll_{\nu}\,x^{-\nu}\ \ \textrm{ for real }\ \nu>0.\end{align}
If $\nu\ge 0$ and $x>1+\nu^2$, then
\begin{align}\label{e25}\begin{split}Y_\nu(x)\ =&\ \sqrt{\frac{2}{\pi x}}\Big(\sin(x-\nu
\pi/2-\pi/4)+O(\frac{1+\nu^2}{x})\Big),
\\  K_\nu(x)\ =&\ \sqrt{\frac{2}{\pi x}}e^{-x}
\Big(1+O(\frac{1+\nu^2}{x})\Big).\end{split}\end{align} Their
derivatives fulfil recurrence relations: for any $\nu\ge 0$,
\begin{align}\label{e26} \frac{d}{dx}\Big(x^{\nu+1}B_{\nu+1}(x)\Big)=\epsilon x^{\nu+1}B_\nu(x)\end{align}
where $\epsilon=1$ for $B=Y$ and $\epsilon=-1$ for $B=K$.

\vskip3mm

\section{The properties of the Gauss sum}
We define
\begin{align*}S(q,a,n)=\sum_{x=1}^qe(\frac{ax^2+nx}{q})\ \ \textrm{ and }\ \
S(q,a)=S(q,a,0).\end{align*}

\begin{lemma}\label{lemma31}Suppose that $(q_1,q_2)=1$. Then we have
\begin{align*}S(q_1q_2,a_1q_2+a_2q_1,n)=S(q_1,a_1q_2^2,n)S(q_2,a_2q_1^2,n).\end{align*}
\end{lemma}

\begin{lemma}\label{lemma32}Suppose that $(2a,q)=1$. Then we have
\begin{align*}S(q,a,n)=e\Big(-\frac{\overline{4a}\, n^2}{q}\Big)\Big(\frac{a}{q}\Big)S(q,1).\end{align*}
Moreover, $|S(q,1)|=q^{1/2}$.
\end{lemma}

\begin{lemma}\label{lemma33}Suppose that $(2,a)=1$. We have
\begin{align*}|S(2^r,a,n)|\le 2^{1+r/2}.\end{align*}
\end{lemma}

\begin{lemma}\label{lemma34}One has
\begin{align*}\Big|\sum_{\substack{a=1\\ (a,p)=1}}^p \Big(\frac{a}{p}\Big)
e\Big(\frac{a}{p}\Big)\Big|=p^{1/2}.\end{align*}
\end{lemma}
The proof of Lemma \ref{lemma31} is elementary, and the proofs of
Lemmas \ref{lemma32}-\ref{lemma34} can be found in Chapter 7 of
\cite{H}. We
define\begin{align*}T(q;n_1,n_2,n_3,m)=\sum_{\substack{a=1\\
(a,q)=1}}^q
S(q,a,n_1)S(q,a,n_2)S(q,a,n_3)e\Big(\frac{\overline{a}\,m}{q}\Big),\end{align*}
and write $T(q)=T(q;0,0,0,0)$.
\begin{lemma}\label{lemma35}Suppose that $(q_1,q_2)=1$. Then one
has
\begin{align*}T(q_1q_2;n_1,n_2,n_3,m)=T(q_1;n_1,n_2,n_3,m)T(q_2;n_1,n_2,n_3,m).\end{align*}
\end{lemma}\begin{proof}The proof is standard. We have
\begin{align*}&\ T(q_1q_2;n_1,n_2,n_3,m)\\=&\sum_{\substack{a=1\\
(a,q_1q_2)=1}}^{q_1q_2}
S(q_1q_2,a,n_1)S(q_1q_2,a,n_2)S(q_1q_2,a,n_3)e\Big(\frac{\overline{a}\,m}{q_1q_2}\Big)
\\=&\sum_{\substack{a=1\\
(a_1,q_1)=1}}^{q_1}\sum_{\substack{a_2=1\\
(a_2,q_2)=1}}^{q_2}\prod_{j=1}^3S(q_1q_2,a_1q_2+a_2q_1,n_j)
e\Big(\frac{\overline{a_1q_2+a_2q_1}\,m}{q_1q_2}\Big).\end{align*}
By Lemma \ref{lemma31}, we obtain
\begin{align*}&\ T(q_1q_2;n_1,n_2,n_3,m)
\\=&\sum_{\substack{a=1\\
(a_1,q_1)=1}}^{q_1}\sum_{\substack{a_2=1\\
(a_2,q_2)=1}}^{q_2}\prod_{j=1}^3\Big(S(q_1,a_1q_2^2,n_j)S(q_2,a_2q_1^2,n_j)\Big)
e\Big(\frac{\overline{a_1q_2^2}\,m}{q_1}\Big)
e\Big(\frac{\overline{a_2q_1^2}\,m}{q_2}\Big)
\\=&\sum_{\substack{a=1\\
(a_1,q_1)=1}}^{q_1}\sum_{\substack{a_2=1\\
(a_2,q_2)=1}}^{q_2}\prod_{j=1}^3\Big(S(q_1,a_1,n_j)S(q_2,a_2,n_j)\Big)
e\Big(\frac{\overline{a_1}\,m}{q_1}\Big)
e\Big(\frac{\overline{a_2}\,m}{q_2}\Big)
\\=&\ T(q_1;n_1,n_2,n_3,m)T(q_2;n_1,n_2,n_3,m).\end{align*}
The desired conclusion is established.
\end{proof}

In view of Lemma \ref{lemma35}, to obtain the upper bound for
$T(q;n_1,n_2,n_3,m)$, it suffices to consider
$T(p^r;n_1,n_2,n_3,m)$.
\begin{lemma}\label{lemma36}One has
\begin{align*}|T(2^r;n_1,n_2,n_3,m)|\le 2^{2+5r/2}.\end{align*}
\end{lemma}\begin{proof}This follows from Lemma \ref{lemma33}.
\end{proof}

\begin{lemma}\label{lemma37}Suppose that $p>2$. Then one has
\begin{align}\label{e31}|T(p^r;n_1,n_2,n_3,m)|\le p^{5r/2}\end{align}
and\begin{align}\label{e32}|T(p;n_1,n_2,n_3,m)|\le
p^{2}.\end{align}
\end{lemma}\begin{proof}By Lemma \ref{lemma32}, for $p>2$ one has
\begin{align*}S(p^r,a,n)=e\Big(-\frac{\overline{4a}\, n^2}{p^r}\Big)\Big(\frac{a}{p^r}\Big)S(p^r,1).\end{align*}
Therefore,
\begin{align*}|T(p^r;n_1,n_2,n_3,m)|\le \sum_{\substack{a=1\\ (a,p^r)=1}}^{p^r}p^{3r/2}\le p^{5r/2}.\end{align*}
This confirms (\ref{e31}). To prove (\ref{e32}), we deduce that
\begin{align*}T(p;n_1,n_2,n_3,m)=&S(p,1)^3\sum_{\substack{a=1\\
(a,p)=1}}^pe\Big(\frac{-\overline{4a}\,(n_1^2+n_2^2+n_3^2-4m)}{p}\Big)\Big(\frac{a}{p}\Big)
\\=&S(p,1)^3\Big(\frac{-1}{p}\Big)\Big(\frac{n_1^2+n_2^2+n_3^2-4m}{p}\Big)\sum_{\substack{b=1\\
(b,p)=1}}^pe\Big(\frac{b}{p}\Big)\Big(\frac{b}{p}\Big).\end{align*}
Then (\ref{e32}) follows from Lemma \ref{lemma32} and Lemma
\ref{lemma34}.
\end{proof}

\begin{lemma}\label{lemma38}Suppose that $q=q_1q_2$ with $(q_1,q_2)=1$, $q_1$ square-free and $q_2$ square-full.
 Then one has
\begin{align*}|T(q;n_1,n_2,n_3,m)|\le 4q_1^2q_2^{5/2}.\end{align*}
\end{lemma}\begin{proof}This follows from Lemmas
\ref{lemma35}-\ref{lemma37} immediately.
\end{proof}
From now on, we assume that $q$ has the decomposition $q=q_1q_2$
with $(q_1,q_2)=1$, where $q_1$ is square-free and $q_2$ is
square-full.

\section{Approximations for $f(\alpha)$ and $h(\alpha)$}

\begin{lemma}\label{lemma41}
Suppose that $(a,q)=1$, $q\le P$ and $|\beta|\le \frac{1}{qP}$.
Then we have
\begin{align}\label{e41}f(\frac{a}{q}+\beta)=\frac{S(q,a)}{q}\int_0^Xe(x^2\beta)dx
+\sum_{-3q/2<b\le 3q/2}S(q,a,b)E(b,q,\beta),\end{align} where
$E(b,q,\beta)$ satisfies
\begin{align}\label{e42}\sum_{-3q/2<b\le 3q/2}|E(b,q,\beta)|\ll \log (q+2).\end{align}
\end{lemma}
\begin{proof}The result is implied by the proof of Theorem 4.1 of
Vaughan \cite{V}. One has
\begin{align*}f(\frac{a}{q}+\beta)=\sum_{x\le X}e(x^2\beta)e(ax^2/q)
=\sum_{x\le X}e(x^2\beta)\sum_{\substack{m=1 \\ m\equiv
x\pmod{q}}}^qe(am^2/q).\end{align*} Then we can obtain
\begin{align}\label{e43}f(\frac{a}{q}+\beta)=\frac{1}{q}\sum_{-q/2<b\le q/2}S(q,a,b)F(b),\end{align}
where \begin{align*}F(b)=\sum_{x\le
X}e(x^2\beta-bx/q).\end{align*} Note that $|2x\beta-b/q|<1$ for
$x\le X$. By Lemma 4.2 in \cite{V},
\begin{align}\label{e44}F(b)=\sum_{-1\le h\le
1}I(b+hq)+E_1(b,q,\beta),\end{align}
where\begin{align*}I(b)=\int_0^X e(x^2\beta-bx/q)dx,\end{align*}
and $E_1(b,q,\beta)$ satisfies
\begin{align}\label{e45}E_1(b,q,\beta)\ll 1.\end{align}
When $b\not=0$ and $0\le x\le X$, $|2x\beta-b/q|\ge
\frac{1}{2}|b/q|$. Then by integration by parts,
\begin{align}\label{e46}I(b)\ll |q/b| \ \ \textrm{ for}\ \ b\not=0.\end{align}
Set $B=3q/2$. On inserting (\ref{e44}) into (\ref{e43}), we obtain
\begin{align*}&f(\frac{a}{q}+\beta)-\frac{S(q,a)}{q}\int_0^Xe(x^2\beta)dx
\\=&\frac{1}{q}\sum_{\substack{-B<b\le B\\ b\not=0}}S(q,a,b)I(b)
+\frac{1}{q}\sum_{-q/2<b\le
q/2}S(q,a,b)E_1(b,q,\beta).\end{align*}
Define\begin{align*}E(b,q,\beta)=\begin{cases}\frac{1}{q}\big(I(b)+E_1(b,q,\beta)\big),
&\ \textrm{ if }\ -q/2<b\le q/2 \ \textrm{ and }\ b\not=0,
\\ \frac{1}{q}E_1(b,q,\beta), &\ \textrm{ if }\ b=0,
\\ \frac{1}{q}I(b), &\ \textrm{ otherwise }.\end{cases}\end{align*}
We can
see\begin{align*}f(\frac{a}{q}+\beta)=\frac{S(q,a)}{q}\int_0^X
e(x^2\beta)dx+\sum_{-B<b\le B}S(q,a,b)E(b,q,\beta).\end{align*} It
follows from (\ref{e45}) and (\ref{e46}) that
\begin{align*}\sum_{-B<b\le B}|E(b,q,\beta)|\ll \log(q+2).\end{align*}
The proof is completed.
\end{proof}

\begin{lemma}\label{lemma42}
Suppose that $(a,q)=1$, $q\le P$ and $|\beta|\le \frac{1}{qP}$.
Then we have
\begin{align}\label{e47}h(\frac{a}{q}+\beta)=&q^{-1}\int (\log x+2\gamma-2\log
q)e(x\beta)w(x)dx
\\ & +\sum_{|n|\not=0}e\Big(-\frac{\overline{a}\, n}{q}\Big)\Delta(n,q,\beta),\notag\end{align} where
$\Delta(n,q,\beta)$ satisfies
\begin{align}\label{e48}\sum_{|n|\not=0}|\Delta(n,q,\beta)|\ll q\log^2 (q+2)+|\beta|^2q^{3/2}X^{7/2}.\end{align}
\end{lemma}\begin{proof}On applying (\ref{e22}) with
$g(x)=e(x\beta)w(x)$, we obtain
\begin{align*}h(\frac{a}{q}+\beta)=&q^{-1}\int (\log x+2\gamma-2\log
q)e(x\beta)w(x)dx \\ &+\sum_{|n|\not=0}e\Big(-\frac{\overline{a}\,
n}{q}\Big)\Delta(n,q,\beta),\end{align*} where
\begin{align*}\Delta(n,q,\beta)=\begin{cases}
-\frac{2\pi \tau(n)}{q}\int_0^\infty
e(x\beta)w(x)Y_0\big(\frac{4\pi\sqrt{xn}}{q}\big)dx,\ \ &\
\textrm{ if }\ n\ge 1,
\\ \frac{4 \tau(|n|)}{q}\int_0^\infty
e(x\beta)w(x)K_0\big(\frac{4\pi\sqrt{x|n|}}{q}\big)dx,\ \ &\
\textrm{ if }\ n\le -1.\end{cases}\end{align*} We use $B$ to
denote the Bessel function $Y$ or $K$. In view of (\ref{e26}), we
deduce from integration by parts that\begin{align*}
&\int_0^{\infty}e(x\beta)w(x)B_0\big(\frac{4\pi\sqrt{x|n|}}{q}\big)dx
\\=&\frac{q}{\epsilon
2\pi\sqrt{|n|}}\int_0^\infty\big(e(x\beta)w(x)\big)'\Big(x^{1/2}B_1\big(\frac{4\pi\sqrt{x|n|}}{q}\big)\Big)dx
\\=&\frac{q^2}{
(2\pi)^2|n|}\int_0^\infty\big(e(x\beta)w(x)\big)''\Big(x
B_2\big(\frac{4\pi\sqrt{x|n|}}{q}\big)\Big)dx.
\end{align*}
By the definition of $w(x)$, we have $w''(x)\ll
|x^{-2}\chi_{S}(x)|$ and $w'(x)\ll |x^{-1}\chi_S(x)|$, where
$\chi_{S}$ denotes the characteristic function over the set $S$
with $S=[1/2,\,1]\cup [3X^2,4X^2]$. Thus,
\begin{align*}\big(e(x\beta)w(x)\big)''\ll & |w''(x)|+|w'(x)\beta|+w(x)|\beta|^2
\\ \ll & x^{-2}\chi_{S}(x)+x^{-1}\chi_S(x)|\beta|+|\beta|^2\chi_{[1/2, 4X^2]}(x)
\\ \ll & x^{-2}\chi_{S}(x)+|\beta|^2\chi_{[1/2,4X^2]}(x).\end{align*}
From above, we deduce that
\begin{align*}&\int_0^\infty e(x\beta)w(x)B_0\big(\frac{4\pi\sqrt{x|n|}}{q}\big)dx
\\ \ll\ &
\frac{q^2}{|n|}\int_{1/2}^{4X^2}\big(x^{-1}\chi_{S}(x)+x|\beta|^2\big)
\Big|B_2\big(\frac{4\pi\sqrt{x|n|}}{q}\big)\Big|dx.\end{align*} By
(\ref{e24}) and (\ref{e25}),
\begin{align}\label{e49}B_2\big(\frac{4\pi\sqrt{x|n|}}{q}\big)=\begin{cases}O(q^{1/2}x^{-1/4}|n|^{-1/4}),\ \
&\textrm{ if }\ x\gg q^2/|n|\, ,\\ O(q^{2}x^{-1}|n|^{-1}),\ \
&\textrm{ if }\ x\ll q^2/|n|\, .
\end{cases}\end{align}
When $|n|>q^2$, we have\begin{align*}&\int_0^\infty
e(x\beta)w(x)B_0\big(\frac{4\pi\sqrt{x|n|}}{q}\big)dx
\\ \ll\ &
\frac{q^2}{|n|}\int_{1/2}^{4X^2}\big(x^{-1}\chi_{S}(x)+x|\beta|^2\big)
q^{1/2}x^{-1/4}|n|^{-1/4}dx \\ \ll \ &
q^{5/2}|n|^{-5/4}+|\beta|^2q^{5/2}X^{7/2}|n|^{-5/4}.\end{align*}
Then we obtain
\begin{align}\label{e410}\sum_{|n|>q^2}|\Delta(n,q,\beta)|\ll q\log^2(q+2)+|\beta|^2X^{7/2}q\log(q+2).
\end{align}
We also have
\begin{align*}
&\int_0^{\infty}e(x\beta)w(x)B_0\big(\frac{4\pi\sqrt{x|n|}}{q}\big)dx
\\=&\frac{q}{\epsilon
2\pi\sqrt{|n|}}\int_0^\infty\big(2\pi i\beta
e(x\beta)w(x)+e(x\beta)w'(x)\big)\Big(x^{1/2}B_1\big(\frac{4\pi\sqrt{x|n|}}{q}\big)\Big)dx
\\=&\Delta_1+\Delta_2+\Delta_3,
\end{align*}
where
\begin{align*}
\Delta_1=&\frac{q}{\epsilon 2\pi\sqrt{|n|}}\int_0^\infty\big(2\pi
i\beta
e(x\beta)w(x)\big)\Big(x^{1/2}B_1\big(\frac{4\pi\sqrt{x|n|}}{q}\big)\Big)dx,
\end{align*}
\begin{align*}
\Delta_2=&\frac{q}{\epsilon
2\pi\sqrt{|n|}}\int_{3X^2}^{4X^2}\big(e(x\beta)w'(x)\big)\Big(x^{1/2}B_1\big(\frac{4\pi\sqrt{x|n|}}{q}\big)\Big)dx,
\end{align*}and
\begin{align*}
\Delta_3=&\frac{q}{\epsilon
2\pi\sqrt{|n|}}\int_{1/2}^{1}\big(e(x\beta)w'(x)\big)\Big(x^{1/2}B_1\big(\frac{4\pi\sqrt{x|n|}}{q}\big)\Big)dx.
\end{align*}
We first handle $\Delta_1$. By (\ref{e26}),
\begin{align*}
\Delta_1=&\frac{q^2}{4\pi^2|n|}\int_0^\infty\big(2\pi i\beta
e(x\beta)w(x)\big)'\Big(xB_2\big(\frac{4\pi\sqrt{x|n|}}{q}\big)\Big)dx.
\end{align*}Then from (\ref{e49}), we get
\begin{align*}\Delta_1 \ll&\frac{q^2}{|n|}\int_{\frac{q^2}{|n|}}^{4X^2}(|\beta|+x|\beta|^2)
q^{1/2}x^{-1/4}|n|^{-1/4}dx
\\ &+\frac{q^2}{|n|}\int_{1/2}^{\frac{q^2}{|n|}} (|\beta|+x|\beta|^2)
q^{2}x^{-1}|n|^{-1}dx
\\ \ll &  |\beta|q^{5/2}X^{3/2}|n|^{-5/4}+|\beta|^2q^{5/2}X^{7/2}|n|^{-5/4}.\end{align*}
We apply (\ref{e26}) again to deduce that
\begin{align*}
\Delta_2=&\ \frac{q^2}{4\pi^2|n|}\int_{3X^2}^{4X^2}
\big(e(x\beta)w'(x)\big)'\Big(xB_2\big(\frac{4\pi\sqrt{x|n|}}{q}\big)\Big)dx
\\ \ll &\ \frac{q^2}{4\pi^2|n|}\int_{3X^2}^{4X^2}
\big(|\beta|+x^{-1}\big)q^{1/2}x^{-1/4}|n|^{-1/4}dx
\\ \ll &\ |\beta|q^{5/2}X^{3/2}|n|^{-5/4}+q^{5/2}X^{-1/2}|n|^{-5/4}.
\end{align*}
When $|n|\le q^2$, by (\ref{e24}) with $\nu=1$, we have
\begin{align*}\Delta_3\ll \frac{q}{\sqrt{|n|}}\int_{1/2}^1x^{-1/2}qx^{-1/2}|n|^{-1/2}dx\ll \frac{q^2}{|n|}.\end{align*}
It follows from above that if $|n|\le q^2$, then
\begin{align*}
\int_0^{\infty}e(x\beta)w(x)B_0\big(\frac{4\pi\sqrt{x|n|}}{q}\big)dx
\ll \frac{q^2}{|n|}+|\beta|^2q^{5/2}X^{7/2}|n|^{-5/4}.
\end{align*}Therefore,
\begin{align}\label{e411}\sum_{|n|\le q^2}|\Delta(n,q,\beta)|\ll q\log^2 (q+2)+|\beta|^2q^{3/2}X^{7/2}.\end{align}
Now (\ref{e48}) follows from (\ref{e410}) and (\ref{e411}).
\end{proof}

It follows from integration by parts together with trivial bounds
that
\begin{align}\label{e412}\int_0^X e(x^2\beta)dx\ll \frac{X}{\sqrt{1+X^2|\beta|}}\end{align}
and
\begin{align}\label{e413}\int (\log x-2\gamma-2\log q)e(x\beta)w(x)dx\ll
 \frac{X^2(\log q+\log X)}{1+X^2|\beta|}.\end{align}

\begin{lemma}\label{lemma43}
Suppose that $q\le P$ and $|\beta|\le \frac{1}{qP}$. For any
$v\in\Z$, we have
\begin{align*}\sum_{\substack{a=1\\ (a,q)=1}}^q
f(\frac{a}{q}+\beta)^3h(-\frac{a}{q}-\beta)e(-\frac{\overline{a}\,v}{q})
=&\frac{1}{q^4}\mathcal{I}(\beta,q)
\sum_{\substack{a=1\\
(a,q)=1}}^qS(q,a)^3e(-\frac{\overline{a}\,v}{q})
\\&+D(q,\beta),\end{align*} where
\begin{align*}\mathcal{I}(\beta,q)=\Big(\int_0^X e(x^2\beta)dx\Big)^3
\int(\log x+2\gamma-2\log q)e(-x\beta)w(x)dx\end{align*} and
$D(q,\beta)$ satisfies
\begin{align}\label{e414}\sum_{q\le P}\int_{|\beta|\le \frac{1}{qP}}
\Big|D(q,\beta)\Big|d\beta\ll X^{2}\log^6X.\end{align}
\end{lemma}\begin{proof}Let\begin{align*}\upsilon(\beta)=\int_0^Xe(x^2\beta)dx\ \textrm{ and }\
\vartheta(\beta,q)=\int(\log x+2\gamma-2\log
q)e(x\beta)w(x)dx.\end{align*} Set $B=3q/2$. Then we introduce
\begin{align*}D_1=\sum_{\substack{-B<b_1\le B\\ -B<b_2\le B \\ -B<b_3\le B}}\sum_{|n|\not=0}
\Big(\prod_{j=1}^3E(b_j,q,\beta)\Big)\Delta(n,q,-\beta)T(q;b_1,b_2,b_3,n-v),\end{align*}
\begin{align*}D_2=\frac{\upsilon(\beta)}{q}\sum_{\substack{-B<b_1\le B\\ -B<b_2\le B }}\sum_{|n|\not=0}
E(b_1,q,\beta)E(b_2
,q,\beta)\Delta(n,q,-\beta)T(q;b_1,b_2,0,n-v),\end{align*}
\begin{align*}D_3=\frac{\upsilon^2(\beta)}{q^2}\sum_{\substack{-B<b_1\le B}}\sum_{|n|\not=0}
E(b_1,q,\beta)\Delta(n,q,-\beta)T(q;b_1,0,0,n-v),\end{align*}
\begin{align*}D_4=\frac{\upsilon^3(\beta)}{q^3}\sum_{|n|\not=0}
\Delta(n,q,-\beta)T(q;0,0,0,n-v),\end{align*}
\begin{align*}D_5=\frac{\vartheta(-\beta,q)}{q}
\sum_{\substack{-B<b_1\le B\\ -B<b_2\le B \\ -B<b_3\le B}}
\Big(\prod_{j=1}^3E(b_j,q,\beta)\Big)T(q;b_1,b_2,b_3,-v),\end{align*}
\begin{align*}D_6=\frac{\upsilon(\beta)\vartheta(-\beta,q)}{q^2}
\sum_{\substack{-B<b_1\le B\\ -B<b_2\le B }} E(b_1,q,\beta)E(b_2
,q,\beta)T(q;b_1,b_2,0,-v),\end{align*} and
\begin{align*}D_7=\frac{\upsilon^2(\beta)\vartheta(-\beta,q)}{q^3}\sum_{\substack{-B<b_1\le B}}
E(b_1,q,\beta)T(q;b_1,0,0,-v).\end{align*} By Lemmas
\ref{lemma41}-\ref{lemma42}, we get
\begin{align*}&\ \sum_{\substack{a=1\\ (a,q)=1}}^q
f(\frac{a}{q}+\beta)^3h(-\frac{a}{q}-\beta)e(-\frac{\overline{a}\,v}{q})
\\=&\ \frac{1}{q^4}\mathcal{I}(\beta,q)
\sum_{\substack{a=1\\
(a,q)=1}}^qS(q,a)^3e(-\frac{\overline{a}\,v}{q})
\\&\ \ \ \ \ +D_1+3D_2+3D_3+D_4+D_5+3D_6+3D_7.\end{align*}
We only handle $D_1,D_4,D_5$ and $D_7$. The estimates for
$D_2,D_3$ and $D_6$ can be handled similarly. By Lemma
\ref{lemma38}, $T(q;b_1,b_2,b_3,n-v)\ll q_1^{2}q_2^{5/2}$. Then by
(\ref{e42}) and (\ref{e48}),\begin{align*}D_1\ll q
_1^2q_2^{5/2}(\log^3X)(q_1q_2\log^2X+|\beta|^2q_1^{3/2}q_2^{3/2}X^{7/2}).\end{align*}
Therefore, we can obtain
\begin{align*}\sum_{q\le P}\int_{|\beta|\le \frac{1}{qP}}|D_1|d\beta\ll X^2\log^6X.\end{align*}
For $D_4$, one has by (\ref{e42}), (\ref{e48}) and (\ref{e412})
that\begin{align*}D_4\ll q
_1^2q_2^{5/2}\frac{X^3}{q_1^{3}q_2^{3}(1+X^2|\beta|)^{3/2}}
(q_1q_2\log^2X+|\beta|^2q_1^{3/2}q_2^{3/2}X^{7/2}).\end{align*}
Then one has
\begin{align*}\sum_{q\le P}\int_{|\beta|\le \frac{1}{qP}}|D_4|d\beta\ll X^2\log^3X.\end{align*}
By (\ref{e42}) and (\ref{e413}), \begin{align*}D_5\ll q
_1^2q_2^{5/2}(\log^4X)\frac{X^2}{q_1q_2(1+X^2|\beta|)}.\end{align*}
It follows that
\begin{align*}\sum_{q\le P}\int_{|\beta|\le \frac{1}{qP}}|D_5|d\beta\ll X^2\log^5X.\end{align*}
By (\ref{e42}), (\ref{e412}) and (\ref{e413}),\begin{align*}D_7\ll
q
_1^2q_2^{5/2}(\log^2X)\frac{X^4}{q_1^{3}q_2^{3}(1+X^2|\beta|)^{2}}.\end{align*}
Thus we have
\begin{align*}\sum_{q\le P}\int_{|\beta|\le \frac{1}{qP}}|D_7|d\beta\ll X^2\log^4X.\end{align*}
Similarly, we can prove
\begin{align*}\sum_{q\le P}\int_{|\beta|\le \frac{1}{qP}}(|D_2|+|D_3|+|D_6|)d\beta\ll X^2\log^5X.\end{align*}
Therefore, (\ref{e414}) is established.
\end{proof}
 \vskip3mm

\section{The Proof of Theorem \ref{theorem}}

The apply the Hardy-Littlewood-Kloosterman circle method to
decompose the integral (\ref{e21}), getting
\begin{align*}S_3(X)=\sum_{q\le P}\sum_{\substack{a=1\\ (a,q)=1}}^q
\int_{\mathcal{B}(q,a)}f(\frac{a}{q}+\beta)^3h(-\frac{a}{q}-\beta)d\beta,\end{align*}
where
\begin{align*}\mathcal{B}(q,a)=\Big[-\frac{1}{q(q+q')},\frac{1}{q(q+q'')}\Big]\end{align*}
with $q'$ and $q''$ satisfying
$$P<q+q',q+q''\le q+P,\ aq'\equiv1\pmod{q},\ aq''\equiv-1\pmod{q}.$$
Note that
\begin{align*}\Big[-\frac{1}{2qP},\frac{1}{2qP}\Big]\subseteq\mathcal{B}(q,a)
\subseteq\Big[-\frac{1}{qP},\frac{1}{qP}\Big].\end{align*} We
exchange the summation over $a$ and the integration over $\beta$
by the standard technique. One may refer to the proof of Lemma 13
of Estermann \cite{Estermann} for this technique (see also Section
3 of Heath-Brown \cite{HB}). We have
\begin{align}\label{e51}&\sum_{\substack{a=1\\ (a,q)=1}}^q
\int_{\beta\in\mathcal{B}(q,a)}\Big(\cdots\Big)d\beta\\
=&\int_{|\beta|\le
\frac{1}{qP}}\sum_{|v|\le P}\sigma(v;\beta,q)\sum_{\substack{a=1\\
(a,q)=1}}^q\Big(\cdots\Big)e\Big(-\frac{\overline{a}\,v}{q}\Big)d\beta\notag\end{align}
for some function $\sigma$ satisfying
\begin{align}\label{e52}\sigma(v;\beta,q)\ll 1/(1+|v|).\end{align}
Therefore,
\begin{align*}S_3(X)=\sum_{q\le P}\int_{|\beta|\le
\frac{1}{qP}}\sum_{|v|\le P}\sigma(v;\beta,q)\sum_{\substack{a=1\\
(a,q)=1}}^qf(\frac{a}{q}+\beta)^3h(-\frac{a}{q}-\beta)e\Big(-\frac{\overline{a}\,
v}{q}\Big)d\beta.\end{align*} In light of Lemma \ref{lemma43},
\begin{align*}S_3(X)=\sum_{q\le P}&\int_{|\beta|\le
\frac{1}{qP}}\sum_{|v|\le P}\sigma(v;\beta,q)\sum_{\substack{a=1\\
(a,q)=1}}^q\frac{S(q,a)^3}{q^4}\mathcal{I}(\beta,q)e\Big(-\frac{\overline{a}\,
v}{q}\Big)d\beta \\ +& \sum_{q\le P}\int_{|\beta|\le
\frac{1}{qP}}\sum_{|v|\le P}\sigma(v;\beta,q)D(q,\beta)d\beta
,\end{align*} where
\begin{align*}\mathcal{I}(\beta,q)=\Big(\int_0^X e(x^2\beta)dx\Big)^3
\int(\log x+2\gamma-2\log q)e(-x\beta)w(x)dx\end{align*} and
$D(q,\beta)$ satisfies
\begin{align*}\sum_{q\le P}\int_{|\beta|\le
\frac{1}{qP}}\Big|D(q,\beta)\Big|d\beta \ll
X^2\log^6X.\end{align*}Then we conclude by (\ref{e52}) that
\begin{align}\label{e53}S_3(X)=\sum_{q\le P}\sum_{\substack{a=1\\ (a,q)=1}}^q
\int_{\beta\in\mathcal{B}(q,a)}\frac{S(q,a)^3}{q^4}\mathcal{I}(\beta,q)d\beta
+ O(X^2\log^7X).\end{align} By (\ref{e51}), we have
\begin{align*}&\sum_{q\le P}\sum_{\substack{a=1\\ (a,q)=1}}^q
\int_{\beta\in\mathcal{B}(q,a)\setminus[-\frac{1}{2qP},\frac{1}{2qP}]}\frac{S(q,a)^3}{q^4}\mathcal{I}(\beta,q)d\beta
\\ =& \sum_{q\le P}
\int_{\frac{1}{2qP}\le |\beta|\le
\frac{1}{qP}}\sum_{|v|\le P}\sigma(v;\beta,q)\sum_{\substack{a=1\\
(a,q)=1}}^q\frac{S(q,a)^3}{q^4}\mathcal{I}(\beta,q)e\Big(-\frac{\overline{a}\,
v}{q}\Big)d\beta \\ =& \sum_{q\le P}\int_{\frac{1}{2qP}\le
|\beta|\le \frac{1}{qP}}\sum_{|v|\le
P}\sigma(v;\beta,q)\frac{T(q;0,0,0,-v)}{q^4}\mathcal{I}(\beta,q)d\beta.\end{align*}
Then we deduce from Lemma \ref{lemma38} and (\ref{e52}) that
\begin{align*}&\sum_{q\le P}\sum_{\substack{a=1\\ (a,q)=1}}^q
\int_{\beta\in\mathcal{B}(q,a)\setminus[-\frac{1}{2qP},\frac{1}{2qP}]}\frac{S(q,a)^3}{q^4}\mathcal{I}(\beta,q)d\beta
\\ \ll & \sum_{|v|\le P}\frac{1}{1+|v|}\sum_{q\le P}\int_{\frac{1}{2qP}\le |\beta|\le
\frac{1}{qP}}
\frac{q_1^2q_2^{5/2}}{q^4}\big|\mathcal{I}(\beta,q)\big|d\beta.\end{align*}
On applying $\mathcal{I}(\beta,q)\ll X^5(\log
X)(1+X^2|\beta|)^{-5/2}$, we can obtain
\begin{align}\label{e54}\sum_{q\le P}\sum_{\substack{a=1\\ (a,q)=1}}^q
\int_{\beta\in\mathcal{B}(q,a)\setminus[-\frac{1}{2qP},\frac{1}{2qP}]}\frac{S(q,a)^3}{q^4}\mathcal{I}(\beta,q)d\beta
\ll X^2\log^3X.\end{align} Moreover, we have
\begin{align}\label{e55}\sum_{q\le P}
\frac{T(q)}{q^4}\int_{|\beta|>
\frac{1}{2qP}}\mathcal{I}(\beta,q)d\beta \ll &\sum_{q\le P}
\frac{|T(q)|}{q^4}X^{3/2}(\log X)q^{3/2} \\ \ll &
X^2\log^2X.\notag\end{align} From (\ref{e53}), (\ref{e54}) and
(\ref{e55}), we get\begin{align}\label{e56}S_3(X)=&\sum_{q\le
P}\sum_{\substack{a=1\\ (a,q)=1}}^q \int_{|\beta|\le
\frac{1}{2qP}}\frac{S(q,a)^3}{q^4}\mathcal{I}(\beta,q)d\beta +
O(X^2\log^7X) \\ =& \sum_{q\le P} \frac{T(q)}{q^4}\int_{|\beta|\le
\frac{1}{2qP}}\mathcal{I}(\beta,q)d\beta+ O(X^2\log^7X)\notag
\\ =& \sum_{q\le P}
\frac{T(q)}{q^4}\int_{-\infty}^{\infty}\mathcal{I}(\beta,q)d\beta+
O(X^2\log^7X)\notag.\end{align} Let
\begin{align*}\mathcal{I}_0(\beta,q)=\Big(\int_0^X e(x^2\beta)dx\Big)^3
\int(\log x+2\gamma-2\log q)e(-x\beta)w_0(x/X^2)dx.\end{align*} By
changing variables, we conclude that
\begin{align}\label{e57}\mathcal{I}_0(\beta,q)=&
X^5(\log X)\mathcal{J}_1(X^2\beta)
\\ &+X^5\Big(\mathcal{J}_2(X^2\beta) + (2\gamma-2\log
q)\mathcal{J}_1(X^2\beta)\Big),\notag\end{align} where
\begin{align*}\mathcal{J}_1(\beta)=\Big(\int_0^1 e(x^2\beta)dx\Big)^3
\int e(-x\beta)w_0(x)dx.\end{align*}
and\begin{align*}\mathcal{J}_2(\beta)=\Big(\int_0^1
e(x^2\beta)dx\Big)^3 \int (\log x)e(-x\beta)w_0(x)dx.\end{align*}
Note that $|\mathcal{I}(\beta,q)-\mathcal{I}_0(\beta,q)|\ll
X^3(\log X)(1+X^2|\beta|)^{-3/2}$. Then we deduce that
\begin{align*}\sum_{q\le P}
\frac{T(q)}{q^4}\int_{-\infty}^{\infty}\Big(\mathcal{I}(\beta,q)-\mathcal{I}_0(\beta,q)\Big)d\beta
\ll & \sum_{q\le P}
\frac{|T(q)|}{q^4}\int_{-\infty}^{\infty}\frac{X^3\log
X}{(1+X^2|\beta|)^{3/2}}d\beta.\end{align*} Therefore, by Lemma
\ref{lemma38}
\begin{align}\label{e58}\sum_{q\le P}
\frac{T(q)}{q^4}\int_{-\infty}^{\infty}\Big(\mathcal{I}(\beta,q)-\mathcal{I}_0(\beta,q)\Big)d\beta
\ll  X\log X.\end{align} By (\ref{e56}), (\ref{e57}) and
(\ref{e58}),
\begin{align*}S_3(X)=&\sum_{q\le P}
\frac{T(q)}{q^4}\int_{-\infty}^{\infty}\mathcal{I}_0(\beta,q)d\beta+
O(X^2\log^7X) \\=& \sum_{q\le P} \frac{T(q)}{q^4}X^3(\log
X)\mathfrak{J}_1+ \sum_{q\le P} \frac{T(q)}{q^4}X^3\mathfrak{J}_2
\\ &\ \ \ \ \ \ \ \ \ \ +
\sum_{q\le P} \frac{T(q)(2\gamma-2\log q)}{q^4}X^3\mathfrak{J}_1
 +O(X^2\log^7X),\end{align*}where
\begin{align*}\mathfrak{J}_1=\int_{-\infty}^{\infty}\mathcal{J}_1(\beta)d\beta\ \textrm{ and }
\
\mathfrak{J}_2=\int_{-\infty}^{\infty}\mathcal{J}_2(\beta)d\beta.\end{align*}
It is easy to verify
\begin{align*}\sum_{q\le P}
\frac{T(q)}{q^4}= \mathfrak{S}_1 +O(X^{-1}\log X),\end{align*} and
\begin{align*}\sum_{q\le P}
\frac{T(q)(2\gamma-2\log q)}{q^4}= \mathfrak{S}_2 +O(X^{-1}\log^2
X),\end{align*}where $\mathfrak{S}_1$ and $\mathfrak{S}_2$ are
given by (\ref{e13}) and (\ref{e14}) respectively. Then we finally
obtain
\begin{align*}S_3(X)=\mathfrak{S}_1\mathfrak{J}_1X^3\log X+
(\mathfrak{S}_1\mathfrak{J}_2+\mathfrak{S}_2\mathfrak{J}_1)X^3+O(X^2\log^7X).\end{align*}
Note that $\mathfrak{S}_1, \mathfrak{S}_2, \mathfrak{J}_1,
\mathfrak{J}_2$ are constants independent of $X$. The proof of
Theorem \ref{theorem} is completed.

 \vskip3mm

\subsection*{Acknowledgements} The author would like to thank the
referee for many helpful comments and suggestions.

\end{document}